\documentclass[smallextended]{svjour3}       
\smartqed  
\usepackage{graphicx}
\usepackage{array}
\usepackage{rotating}
\usepackage{multirow}
\usepackage{amssymb}
\usepackage{amsmath}
\usepackage{algorithm}
\usepackage{algorithmic}
\newcommand{\conv}{\mathrm{conv}} 

\newcommand{\link}{\mathrm{link}}

\newcommand{\dom}{\mathrm{dom}}
\newcolumntype{C}{>{\centering\arraybackslash}m{0.5cm}<{}}
\begin{document}

\title{The flip-graph of the $4$-dimensional cube is connected}


\author{Lionel Pournin}

\institute{L. Pournin \at
             EFREI, 30-32 avenue de la R{\'e}publique, 94800 Villejuif, France\\
             \email{lionel.pournin@efrei.fr}\\
             EPFL - {\'E}cole Polytechnique F{\'e}d{\'e}rale de Lausanne, 1015 Lausanne, Switzerland\\
                           \email{lionel.pournin@epfl.ch}
}


\maketitle

\begin{abstract}
Flip-graph connectedness is established here for the vertex set of the 4-dimensional cube. It is found as a consequence that this vertex set has $92\,487\,256$ triangulations, partitioned into $247\,451$ symmetry classes.


\end{abstract}

\section{Introduction}
\label{se.0}

It is well known that the vertex set of the $3$-dimensional cube has exactly $74$ triangulations, all regular, partitioned into $6$ symmetry classes \cite{deL96,DRS10}. The case of the $4$-dimensional cube turns out to be significantly more complicated. The first non-regular triangulation of its vertex set has been found almost two decades ago \cite{deL96}, and the total number of such triangulations was, up to now, unknown. The reason for this is that the only triangulation enumeration method efficient enough to be tractable in the case of the $4$-dimensional cube actually consists in exploring the flip-graph of its vertex set \cite{DRS10}. In this perspective, completely enumerating the triangulations of the vertex set of the $4$-dimensional cube is a task conditioned to the connectedness of this graph, which remained an open problem until now \cite{DRS10}.

The $4$-dimensional cube is identified hereafter with the polytope $[0,1]^4$ and its vertices with the elements of $\{0,1\}^4$. It is proven in this paper that the flip-graph of $\{0,1\}^4$ is connected, and found as a consequence, that the vertex set of the $4$-dimensional cube has $92\,487\,256$ triangulations, partitioned into $247\,451$ symmetry classes. Some of the results in the paper require computer assistance. These computations were mostly done using TOPCOM \cite{Ram02}.

The proof consists in finding paths in the flip-graph of $\{0,1\}^4$ from any triangulation to a corner-cut triangulation \cite{DRS10}. To this end, one needs to perform a sequence of flips that introduces new corner simplices in a triangulation $T$ of $\{0,1\}^4$. Consider the $3$-dimensional triangulation $U$ obtained when intersecting $T$ with a hyperplane that cuts the corner of $[0,1]^4$ with apex $x$. It is shown that one can choose $x$ so that, up to an isometry, $U$ belongs to a well characterized set of $1\,588$ triangulations. In addition, conditions are given such that flips in these triangulations carry over to $T$ and reduce the star of $\{x\}$ in $T$ when they are performed. Hence, the problem is narrowed to checking a simple property regarding flips over the above mentioned $1\,588$ triangulations. This verification is performed within minutes using a computer, and does not require any arithmetic operations.

The paper is organized as follows. In section \ref{se.2}, formal definitions of the mathematical objects involved in the problem are stated. Section \ref{se.3} presents preliminary results about the triangulations of $\{0,1\}^4$. It turns out, in particular, that the corner-cut triangulations of $\{0,1\}^4$ are regular. In section \ref{se.4}, it is shown how the flips in a triangulation $T$ of $\{0,1\}^4$ are related to flips in the triangulations obtained by intersecting $T$ with a hyperplane that cuts a corner of the $4$-dimensional cube. Using this, it it is proven in section \ref{se.5} that every triangulation of $\{0,1\}^4$ can made regular by performing a sequence of flips, and connectedness follows for the flip-graph of $\{0,1\}^4$.

\section{Preliminary definitions}
\label{se.2}

In this paper, a {\it point configuration} is a finite subset of $\mathbb{R}^4$. The elements of a point configuration are referred to as its {\it vertices}. Consider a point configuration $\mathcal{A}$. A simplicial complex on $\mathcal{A}$ is a set of affinely independent subsets of $\mathcal{A}$ whose convex hulls collectively form a polyhedral complex. The elements of a simplicial complex $C$ are called its {\it faces} and their vertices are also referred to as the {\it vertices of $C$}. The faces of $C$ that are maximal for the inclusion are called its {\it maximal faces}. The {\it domain }Êof a simplicial complex $C$ on $\mathcal{A}$, denoted by $\dom(C)$, is the union of the convex hulls of its faces:
$$
\dom(C)=\bigcup_{t\in{C}}\conv(t)
$$

A triangulation of $\mathcal{A}$ is a simplicial complex on $\mathcal{A}$ whose domain is precisely $\conv(\mathcal{A})$. Consider a triangulation $T$ of $\mathcal{A}$. For any face $t$ of $T$, the {\it link} of $t$ in triangulation $T$ is the following set:
$$
\link_T(t)=\{s\in{T}:s\cup{t}\in{T},\,s\cap{t}=\emptyset\}\mbox{.}
$$

If $X$ and $Y$ are two finite subsets of $\mathbb{R}^4$, denote by $X\star{Y}$ the set obtained by taking the union of each element of $X$ with each element of $Y$:
$$
X\star{Y}=\{x\cup{y}:(x,y)\in{X\times{Y}}\}\mbox{.}
$$

Further denote by $\mathcal{P}(S)$ the power set of any set $S$. If $t$ is a face of a triangulation $T$, these notations can be used to define the {\it star} of $t$ in $T$, that is, the subset of the faces of $T$ whose union with $t$ still belongs to $T$:
$$
\mathrm{star}_T(t)=\mathcal{P}(t)\star\link_T(t)\mbox{.}
$$

A circuit $z\subset\mathcal{A}$ is an affinely dependent set whose every proper subset is affinely independent. According to Radon's partition theorem \cite{Rad21}, a circuit $z$ admits a unique partition into two subsets $z^-$ and $z^+$ whose convex hulls are non-disjoint. Because of the unicity of this partition, the convex hulls of $z^-$ and $z^+$ even have non-disjoint relative interiors. As a consequence, $z^-$ and $z^+$ cannot both belong to the same triangulation of $z$. In fact, $z$ admits exactly two triangulations $\tau^-$ and $\tau^+$ that can be defined as follows:
$$
\tau^-=\{t\subset{z}:z^+\not\subset{t}\}\mbox{ and }\tau^+=\{t\subset{z}:z^-\not\subset{t}\}\mbox{.}
$$

Note in particular that $z^-$ and $z^+$ are respectively a face of $\tau^-$ and a face of $\tau^+$. Moreover, the maximal faces of $\tau^-$ and of $\tau^+$ always respectively admit $z^-$ and $z^+$ as subsets. If $T$ is a triangulation of $\mathcal{A}$, circuit $z$ is said to be {\it flippable in $T$} if there exist two subsets $\tau$ and $\lambda$ of $T$ so that $\tau$ is a triangulation of $z$ and every maximal face of $\tau$ admits $\lambda$ as its link in $T$. By symmetry, one can assume without loss of generality that $\tau$ is equal to $\tau^-$. In this case, the following set is a triangulation of $\mathcal{A}$ distinct from $T$ (see \cite{DRS10,San00}):
$$
\mathfrak{F}(z,T)=[T\setminus(\lambda\star\tau^-)]\cup(\lambda\star\tau^+)\mbox{.}
$$

In other words, $\mathfrak{F}(z,T)$ is obtained replacing in triangulation $T$ the faces of $\lambda\star\tau^-$ by the faces of $\lambda\star\tau^+$. The operation of transforming $T$ into $\mathfrak{F}(z,T)$ is called a {\it flip}. Observe that, according to the definitions of $\tau^-$ and $\tau^+$, triangulations $T$ and $\mathfrak{F}(z,T)$ respectively contain $z^-$ and $z^+$. Moreover, all the faces that are removed from $T$ when $z$ is flipped in this triangulation admit $z^-$ as a subset. This property will be useful hereafter.

The interested reader is referred to \cite{DRS10} for further definitions and results on the subject of triangulations and flips.

The flip-graph of $\mathcal{A}$, denoted by $\gamma(\mathcal{A})$ in the following, is the graph whose vertices are the triangulations of $\mathcal{A}$ and whose edges correspond to flips. The main result in this paper is the connectedness of $\gamma(\{0,1\}^4)$. On the way to establishing this result, the first step is to prove that the corner-cut triangulations \cite{DRS10} of $\{0,1\}^4$ are regular. The notion of regularity can be defined using {\it height functions} that is, real-valued maps on point configurations. A triangulation $T$ of a point configuration $\mathcal{A}$ is regular if there exists a height function $w:\mathcal{A}\rightarrow\mathbb{R}$ so that for all $t\in{T}$, some affine map $\xi:\mathbb{R}^4\rightarrow\mathbb{R}$ is smaller than $w$ on $\mathcal{A}\setminus{t}$ and coincides with $w$ on $t$. It is well known that the graph $\rho(\mathcal{A})$ induced by regular triangulations in $\gamma(\mathcal{A})$ is connected. Indeed, $\rho(\mathcal{A})$ is isomorphic to the $1$-skeleton of the so-called secondary polytope \cite{Gel90,Gel94}. This connectedness property will be the last argument used in the proof.

\section{First properties of $\{0,1\}^4$ and of its triangulations}
\label{se.3}

Denote the vertices of $\{0,1\}^4$ by letters $a$ to $p$ according to their positions as columns of the following matrix. The rows of this matrix show the coordinates of these points along the vectors of the canonical basis $(u_1,u_2,u_3,u_4)$ of $\mathbb{R}^4$:
$$
\begin{array}{lr}
\begin{array}{r}
\mbox{ }\\
u_1\\
u_2\\
u_3\\
u_4\\
\end{array}
&
\begin{array}{c}
\begin{array}{CCCCCCCCCCCCCCCC}
$a$ & $b$ & $c$ & $d$ & $e$ & $f$ & $g$ & $h$ & $i$ & $j$ & $k$ & $l$ & $m$ & $n$ & $o$ & $p$
\end{array}\\
\left[
\begin{array}{CCCCCCCCCCCCCCCC}
0 & 1 & 0 & 1 & 0 & 1 & 0 & 1 & 0 & 1 & 0 & 1 & 0 & 1 & 0 & 1\\
0 & 0 & 1 & 1 & 0 & 0 & 1 & 1 & 0 & 0 & 1 & 1 & 0 & 0 & 1 & 1\\
0 & 0 & 0 & 0 & 1 & 1 & 1 & 1 & 0 & 0 & 0 & 0 & 1 & 1 & 1 & 1\\
0 & 0 & 0 & 0 & 0 & 0 & 0 & 0 & 1 & 1 & 1 & 1 & 1 & 1 & 1 & 1\\
\end{array}
\right]
\end{array}
\end{array}
$$

Observe that the squared distance between two points of $\{0,1\}^4$ is an integer from $0$ to $4$ and that the boundary complex of $[0,1]^4$ is made up of cubes whose dimensions also range from $0$ to $4$. It turns out that, if the squared distance of two points $x$ and $y$ of $\{0,1\}^4$ is an integer $q$, then some $q$-dimensional face of $[0,1]^4$ admits $\{x,y\}$ as one of its diagonals. Hence, if $T$ is a triangulation of $\{0,1\}^4$, then every edge of its $1$-skeleton $\sigma(T)$ is a diagonal of some cube in the face complex of $[0,1]^4$. For any integer $q$, denote by $\sigma_q(T)$ the graph obtained by removing from $\sigma(T)$ all the edges whose squared length is not equal to $q$.

The $f$-vector of $[0,1]^4$ is $(1,16,32,24,8,1)$. Since distinct diagonals of a cube are never found in a same triangulation of this cube, then:

\begin{proposition}\label{pr.2}
It $T$ is a triangulation of $\{0,1\}^4$, $\sigma_2(T)$ admits at most $24$ edges, $\sigma_3(T)$ admits at most $8$ edges, and $\sigma_4(T)$ admits at most $1$ edge.
\end{proposition}

Now, consider a point $x\in\{0,1\}^4$. Further consider the points in $\{0,1\}^4$ whose squared distance to $x$ is $1$. There are exactly four such points that are precisely the vertices adjacent to $x$ in the $1$-skeleton of $[0,1]^4$. The set $\kappa(x)$ made up of these four points and of point $x$ is called the {\it corner simplex} of $\{0,1\}^4$ with apex $x$. Observe that if $y\in\{0,1\}^4$ is a point whose distance to $x$ is larger than $1$, then the convex hull of $\{x,y\}$ and a face of $\conv(\kappa(x)\setminus\{x\})$ have non-disjoint relative interiors. Therefore, if $T$ is a triangulation of $\{0,1\}^4$ that contains $\kappa(x)$, then $x$ is an isolated vertex in the three graphs $\sigma_2(T)$, $\sigma_3(T)$, and $\sigma_4(T)$. Inversely, if $T$ is a triangulation of $\{0,1\}^4$  so that $x$ is isolated in $\sigma_2(T)$, $\sigma_3(T)$, and $\sigma_4(T)$, then the only $4$-dimensional face of $T$ that possibly contains $x$ is necessarily $\kappa(x)$. Hence:

\begin{proposition}\label{pr.3}
Let $x$ be a point of $\{0,1\}^4$. A triangulation $T$ of $\{0,1\}^4$ contains $\kappa(x)$ if and only if $x$ is isolated in $\sigma_2(T)$, in $\sigma_3(T)$, and in $\sigma_4(T)$.
\end{proposition}

Now denote by $E$ the set made up of all the vertices of $\{0,1\}^4$ whose sum of coordinates in the canonical basis of $\mathbb{R}^4$ is even. Further denote by $O$ the set obtained removing the elements of $E$ from $\{0,1\}^4$. The vertices of $E$ and $O$ can be found from the above matrix:
$$
E=\{a,d,f,g,j,k,m,p\}\mbox{ and }O=\{b,c,e,h,i,l,n,o\}\mbox{.}
$$

Observe that the squared distance of two vertices of $\{0,1\}^4$ is even if and only if they both belong to the same of these two sets. In other words, $\{E,O\}$ is the only partition of $\{0,1\}^4$ into two subsets whose pairs of vertices have even squared distance. Further note that $E$ and $O$ can be obtained from one another by an isometry and that their convex hulls are $4$-dimensional cross polytopes \cite{DRS10}. In fact, if $s\in\{E,O\}$, one obtains the interior of $\conv(s)$ by cutting off from $[0,1]^4$ the convex hulls of the corner simplices with apex in $\{0,1\}^4\setminus{s}$. Since cross polytopes are simplicial, it follows that one can build a triangulation of $\{0,1\}^4$ by taking the union of a triangulation of $s$ with the power sets of every such corner simplices. The resulting triangulations of $\{0,1\}^4$ are called {\it corner-cut}. It is well-known that $\{0,1\}^4$ admits precisely eight corner-cut triangulations, and that each of the eight diagonals of $[0,1]^4$ is found in exactly one of these triangulations \cite{DRS10}. It turns out that every such triangulation of $\{0,1\}^4$ is regular:

\begin{lemma}\label{lem.1}
All the corner-cut triangulations of $\{0,1\}^4$ are regular.
\end{lemma}
\begin{proof}
Consider a corner-cut triangulation $T$ of $\{0,1\}^4$. By symmetry, one can assume without loss of generality that $T$ contains all the corner simplices of $\{0,1\}^4$ whose apex belongs to $O$. Consider the set $T'$ made up of the subsets of $E$ that are faces of $T$:
$$
T'=\{t\in{T}:t\subset{E}\}\mbox{.}
$$

As mentioned above, $T'$ is a triangulation of $E$ (that is, of the vertex set of a $4$-dimensional cross polytope). 
The following diagonals of the $4$-dimensional cube partition $E$ into four subsets:
$$
\{a,p\}\mbox{, }\{d,m\}\mbox{, }\{f,k\}\mbox{, and }\{g,j\}\mbox{.}
$$

It follows from proposition \ref{pr.2} that $T'$ contains at most one of these diagonals. Moreover, since each maximal face of $T'$ has $5$ vertices, any such face necessarily admits one of these diagonals as a subset. This proves that one of the four diagonals enumerated above is found as a subset of every maximal face of $T'$. By symmetry, one can assume that this diagonal is $\{a,p\}$. Now consider the height function $w:\{0,1\}^4\rightarrow\mathbb{R}$ that maps points $a$ and $p$ to $0$, the rest of $E$ to $1$, and every vertex of $O$ to $2$.

Consider a maximal face $t$ of $T$. As already mentioned, either $t$ is a corner simplex of $\{0,1\}^4$ whose apex is in $O$ or $t$ is a maximal face of $T'$. First assume that $t$ is a maximal face of $T'$. In this case, $t$ contains three vertices of $E$, in addition to $a$ and $p$. Since $t$ is affinely independent, each of the three diagonals $\{d,m\}$, $\{f,k\}$, and $\{g,j\}$ therefore admits exactly one of its vertices in $t$. By symmetry, one can assume that these three vertices are $d$, $f$, and $g$, that is $t=\{a,d,f,g,p\}$. Consider the vector:
$$
v_1=\frac{1}{2}(u_1+u_2+u_3)-\frac{3}{2}u_4\mbox{,}
$$
and the affine map $\xi_1$ that projects each vector $x\in\mathbb{R}^4$ to $x\cdot{v_1}$. Using the coordinates given in the matrix above, one can check that $\xi_1$ coincides with $w$ on $t$, and that $\xi_1$ is smaller than $w$ on the rest of $\{0,1\}^4$. Now assume that $t$ is a corner simplex of $\{0,1\}^4$ whose apex is in $O$. Here again, it can be assumed, using the symmetries of $\{0,1\}^4$ that $t$ is the corner simplex of $\{0,1\}^4$ with apex $b$, that is $t=\{a,b,d,f,j\}$. Consider the vector:
$$
v_2=2u_1-(u_2+u_3+u_4)\mbox{,}
$$
and the affine map $\xi_2$ that projects each vector $x\in\mathbb{R}^4$ to $x\cdot{v_2}$. Again, using the coordinates given in the matrix above, one can check that $\xi_2$ coincides with $w$ on $t$, and that $\xi_2$ is smaller than $w$ on the rest of $\{0,1\}^4$. As two simplicial complexes with the same maximal faces are necessarily identical, this shows that triangulation $T$ is regular.
\qed
\end{proof}

It is proven in section \ref{se.5} that any triangulation of $\{0,1\}^4$ is connected in $\gamma(\{0,1\}^4)$ to a corner-cut triangulation. According to lemma \ref{lem.1}, the connectedness of $\gamma(\{0,1\}^4)$ will then naturally follow from the connectedness of $\rho(\{0,1\}^4)$. The proof that a triangulation $T$ of $\{0,1\}^4$ can be transformed by flips into a corner-cut triangulation will require a careful study of the way $T$ decomposes the corners of $[0,1]^4$. In particular, bounds on the degrees of the vertices of graphs $\sigma_q(T)$  will be needed. These bounds are provided by the two following lemmas.

\begin{lemma}\label{lem.2}
For any triangulation $T$ of $\{0,1\}^4$, point configurations $E$ and $O$ each admit at least four vertices whose degree in $\sigma_3(T)$ is at most $1$.
\end{lemma}
\begin{proof}
Consider an element $s$ of $\{E,O\}$. Assume that more than four points of $s$ have degree at least $2$ in $\sigma_3(T)$. In this case, the sum over $s$ of the degrees in $\sigma_3(T)$ is greater than $8$. According to proposition \ref{pr.2} though, $\sigma_3(T)$ contains at most eight edges. Hence, at least one edge of $\sigma_3(T)$ has its two vertices in $s$. Since the squared distance between two points of $s$ is even and since all the edges in $\sigma_3(T)$ have squared length $3$, one obtains a contradiction. This proves that at most four vertices of $s$ have degree at least $2$ in $\sigma_3(T)$. Since $s$ has cardinality $8$, the desired result follows.\qed
\end{proof}

Consider an element $s$ of $\{E,O\}$. The bound given in lemma \ref{lem.2} on the degrees in graph $\sigma_3(T)$ does not hold for every point of $s$. Under a condition on the way $T$ triangulates the corners of $\{0,1\}^4$, a bound on the degrees in graph $\sigma_2(T)$ can be obtained that holds for every point of $s$.

\begin{lemma}\label{lem.3}
Let $T$ be a triangulation of $\{0,1\}^4$ and $s$ an element of $\{E,O\}$. If at least four corner simplices of $\{0,1\}^4$ with apex in $s$ are contained in $T$, then all the vertices of $s$ have degree at most $3$ in $\sigma_2(T)$.
\end{lemma}
\begin{proof}
Assume that at least four corner simplices of $\{0,1\}^4$ with apex in $s$  are found in triangulation $T$. According to proposition \ref{pr.3}, these four apices are isolated in $\sigma_2(T)$. As $s$ contains exactly eight points, then at most four vertices of $s$ are not isolated in $\sigma_2(T)$. Now recall that the squared distance between a vertex of $s$ and a vertex of $\{0,1\}^4\setminus{s}$ is odd. As a consequence, two adjacent points in $\sigma_2(T)$ are necessarily both in $s$ or both in $\{0,1\}^4\setminus{s}$. Since at most four vertices of $s$ are not isolated $\sigma_2(T)$, every vertex of $s$ is adjacent to at most three points in this graph, which completes the proof. \qed
\end{proof}

\section{Flips in the corners of the $4$-dimensional cube}
\label{se.4}

Consider a point $x$ in $\{0,1\}^4$ and observe that $\kappa(x)\setminus\{x\}$ is a (regular) tetrahedron. The affine hull of this tetrahedron is an affine hyperplane $H$ of $\mathbb{R}^4$ that admits $p-2x$ as a normal vector. In addition, $H$ separates $x$ from the vertices of $\{0,1\}^4$ that do not belong to $\kappa(x)$. As a consequence, every vertex $y$ of $\{0,1\}^4$ other than $x$ can be centrally projected on $H$ with respect to point $x$. This projection will be denoted by $y/x$:
$$
y/x=x+\frac{4}{(y-x)\cdot(p-2x)}(y-x)\mbox{.}
$$

Now let $s$ be a subset of $\{0,1\}^4$. The following set is called the {\it homogeneous contraction} of $s$ at point $x$:
$$
s/x=\{y/x:y\in{s\setminus\{x\}}\}\mbox{.}
$$

This terminology, introduced in \cite{DRS10}, is due to the similarity of this notion with the contraction of oriented matroids. The notation given in \cite{DRS10} is made up of two fraction bars instead of just one, in order to distinguish these two types of contractions. This precaution is waived here because only homogeneous contractions are used. Note that homogeneous contractions are not generally required to lie in a fixed affine space \cite{DRS10}. Here, $s/x$ is explicitly placed in the affine hull of $\kappa(x)\setminus\{x\}$. This will be convenient in the following, especially for geometric interpretations. Observe that $\{0,1\}^4$ does not contain any three collinear points. As a consequence, the homogeneous contraction at point $x$ induces a bijection from $\{0,1\}^4\setminus\{x\}$ onto $\{0,1\}^4/x$. 

Now consider two elements $x$ and $y$ of $\{0,1\}^4$. By symmetry, point configurations $\{0,1\}^4/x$ and $\{0,1\}^4/y$ are isometric.
 \begin{figure}
\begin{centering}
\includegraphics{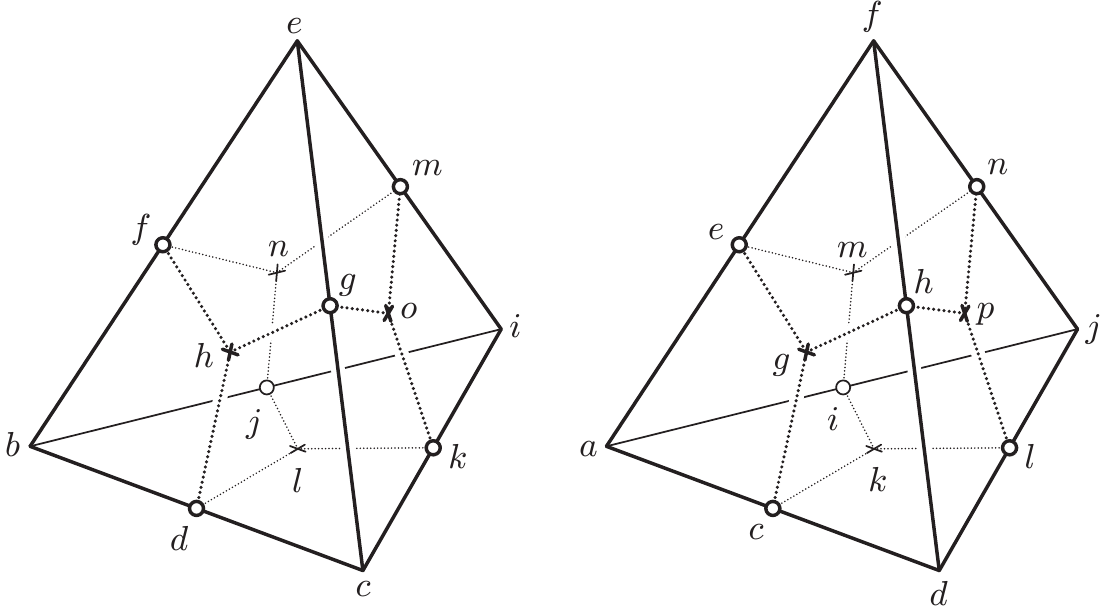}
\caption{Point configurations $\{0,1\}^4/a$ (left) and $\{0,1\}^4/b$ (right) where the point at the center of each tetrahedron has been omitted. Every point is labeled as the vertex of $\{0,1\}^4$ it is projected from by the homogeneous contraction. The only purpose of the dotted lines is to make it clear which point lies at the center of each triangular face.}\label{fig.1}
\end{centering}
\end{figure}
For instance, the homogeneous contractions of $\{0,1\}^4$ at points $a$ and $b$ are depicted in figure \ref{fig.1}. Note that vertices $p/a$ and $o/b$, that should be found in this figure at the center of the two tetrahedra have been omitted for the sake of clarity. 

Homogeneous contractions at a point $x\in\{0,1\}^4$ can be generalized as follows to any set $S$ whose elements are subsets of $\{0,1\}^4$:
$$
S/x=\{s/x:s\in{S}\}\mbox{.}
$$

Now denote by $Z(x)$ the set of all the circuits $z\subset\{0,1\}^4$ that contain point $x$. The Radon partition of any circuit $z$ in $Z(x)$ will be denoted hereafter by $\{\varepsilon_x^-(z),\varepsilon_x^+(z)\}$ with the convention that $x$ belongs to $\varepsilon_x^-(z)$. According to the next lemma, the homogeneous contraction at point $x$ of any such circuit remains a circuit. This result follows from the straightforward property that when $x$ is contained in a subset $s$ of $\{0,1\}^4$, the affine hull of $s/x$ is the intersection of the affine hull of $s$ with that of $\kappa(x)\setminus\{x\}$.

\begin{lemma}\label{lem.4}
Let $x$ be a vertex of $\{0,1\}^4$. For every $z\in{Z(x)}$, $z/x$ is a circuit whose Radon partition is $\{\varepsilon_x^-(z),\varepsilon_x^+(z)\}/x$.
\end{lemma}
\begin{proof}
Let $z\in{Z(x)}$ be a circuit. Consider a subset $s$ of $z$ that contains $x$ and recall that the homogeneous contraction at point $x$ induces an bijection from $\{0,1\}^4\setminus\{x\}$ onto $\{0,1\}^4/x$. As a consequence, $s/x$ has exactly one element less than $s$. Further recall that the affine hull of $\kappa(x)\setminus\{x\}$ is an hyperplane of $\mathbb{R}^4$ and denote by $H$ this hyperplane. Since $x$ belongs to $s$, the affine hull of $s/x$ is the intersection of $H$ with the affine hull of $s$. Hence the homogeneous contraction at point $x$ not only decreases by one the cardinality of $s$ but also the dimension of its affine hull. As a consequence, $s/x$ is affinely dependent if and only if $s$ is affinely dependent. Replacing $s$ by $z$ in this assertion, one obtains that $z/x$ is affinely dependent. Further note that every proper subset of $z/x$ can be obtained as the homogeneous contraction at point $x$ of a proper subset of $z$ that contains $x$. As $z$ is a circuit, all its proper subsets are affinely independent. According to the above assertion, this property carries over to the proper subsets of $z/x$, which proves that $z/x$ is a circuit.

Respectively denote $\varepsilon_x^-(z)$ and $\varepsilon_x^+(z)$ by $z^-$ and by $z^+$. Recall that the convex hulls of $z^-$ and $z^+$ have non-disjoint relative interiors. As in addition, $x$ belongs to $z^-$ and does not belong to $z^+$, there exists a $1$-dimensional affine subspace $N$ of $\mathbb{R}^4$ that contains $x$ and that intersects the relative interiors of $\conv(z^-\setminus\{x\})$ and of $\conv(z^+)$. Moreover, $z^-/x$ and $z^+/x$ are the central projections on $H$ with respect to point $x$ of respectively $z^-\setminus\{x\}$ and $z^+$. Hence, the convex hulls of $z^-/x$ and $z^+/x$ both contain the intersection of $N$ with $H$, proving that these convex hulls are non-disjoint. Finally, since the homogeneous contraction at point $x$ induces a bijection from $\{0,1\}^4\setminus\{x\}$ onto $\{0,1\}^4/x$, sets $z^-/x$ and $z^+/x$ partition $z/x$. Hence, by definition, the Radon partition of $z/x$ is $\{z^-,z^+\}/x$, which completes the proof. \qed
\end{proof}

Note that, while this lemma is stated here for $\{0,1\}^4$, a similar result holds in general for arbitrary point configurations of any dimension.

Let $x$ be a point of $\{0,1\}^4$ and $T$ a triangulation of $\{0,1\}^4$. Consider the set obtained as the homogeneous contraction of $\link_T(\{x\})$ at point $x$. This set can be alternatively obtained by intersecting the convex hulls of the faces of $\mathrm{star}_T(\{x\})$ with the affine hull $H$ of $\kappa(x)\setminus\{x\}$ and by subsequently taking the vertex sets of these intersections. As a consequence, the convex hulls of the faces of $\link_T(\{x\})/x$ collectively form a polyhedral complex. In addition, every convex combination of $\{0,1\}^4/x$ is found in the intersection of $H$ with the convex hull of some face of $\mathrm{star}_T(\{x\})$. Hence $\link_T(\{x\})/x$ is a triangulation of $\{0,1\}^4/x$. Note in particular that the homogeneous contraction at point $x$ not only induces a bijection from $\{0,1\}^4\setminus\{x\}$ onto $\{0,1\}^4/x$ but also an isomorphism from $\link_T(\{x\})$ onto triangulation $\link_T(\{x\})/x$. This triangulation and its properties are studied in remainder of the section. 

Consider a circuit $z$ in $Z(x)$. A consequence of lemma \ref{lem.4} is that, if $z$ is flippable in $T$, then $z/x$ is flippable in $\link_T(\{x\})/x$. While the converse implication is not true in general, the next theorem states that it becomes so if one further assumes that $\link_T(\{x\})/x$ contains $\varepsilon_x^-(z)/x$.
Indeed, in this case, $\varepsilon_x^-(z)$ must be a face of $T$. Hence, if $z$ can be flipped in $T$, all the faces removed by this flip necessarily admit $\varepsilon_x^-(z)$ as a subset. As $x$ is a vertex of $\varepsilon_x^-(z)$, every such face belongs to the star of $\{x\}$ in $T$ and can be reconstructed from $\link_T(\{x\})/x$. It can be proven, as a consequence, that the flippability of $z/x$ in triangulation $\link_T(\{x\})/x$ carries over to the flippability of $z$ in triangulation $T$. Note that the result fails if $\varepsilon_x^-(z)$ is replaced by $\varepsilon_x^+(z)$ precisely because $x$  does not belong to $\varepsilon_x^+(z)$. In this case, $z\setminus\{x\}$ would be affected if $z$ were flipped in $T$. However, this simplex cannot be reconstructed from $\link_T(\{x\})/x$ because it does not belong to the star of $\{x\}$ in $T$.

\begin{theorem}\label{thm.1}
Let $x$ be a vertex of $\{0,1\}^4$, $T$ a triangulation of $\{0,1\}^4$, and $z\in{Z(x)}$ a circuit. If $z/x$ is flippable in $\link_T(\{x\})/x$ and $\varepsilon_x^-(z)/x$ belongs to $\link_T(\{x\})/x$, then $z$ is flippable in $T$ and $\varepsilon_x^-(z)$ belongs to $T$.
\end{theorem}
\begin{proof}
First assume that $z/x$ is flippable in $\link_T(\{x\})/x$. It follows that some triangulation of $z/x$ is a subset of $\link_T(\{x\})/x$. Denote:
$$
\tau=\{t\subset{z}:t/x\in\link_T(\{x\})/x\}\mbox{.}
$$

By construction, $\tau/x$ is precisely the triangulation of $z/x$ found as a subset of $\link_T(\{x\})/x$. In fact $\tau$ is a subset of $T$. Indeed, consider a face $t$ of $\tau$. By definition, $t/x$ belongs to $\link_T(\{x\})/x$. Hence, some face $s$ of $\link_T(\{x\})$ satisfies $s/x=t/x$. Now recall that the homogeneous contraction at vertex $x$ induces a bijection from $\{0,1\}^4\setminus\{x\}$ onto $\{0,1\}^4/x$. As a consequence $s$ is necessarily a subset of $t$. In fact $t$ is either equal to $s$ or to $s\cup\{x\}$. In both cases, $t$ belongs to $T$ because $s\in\link_T(\{x\})$ and $s\cup\{x\}\in\mbox{star}_T(\{x\})$.

Further assume that $\link_T(\{x\})/x$ contains $\varepsilon_x^-(z)/x$. As $\varepsilon_x^-(z)/x$ is a subset of $z/x$, then it necessarily belongs to $\tau/x$. According to lemma \ref{lem.4}, the Radon partition of $z/x$ is made up of $\varepsilon_x^-(z)/x$ and of $\varepsilon_x^+(z)/x$. By definition of the triangulations of a circuit, one therefore obtains:
\begin{equation}\label{thm.1.eq.0}
\tau/x=\{s\subset{z/x}:\varepsilon_x^+(z)/x\not\subset{s}\}\mbox{.}
\end{equation}

Let $t$ be a subset of $z$ that does not admit $\varepsilon_x^+(z)$ as a subset. It can be proven, using (\ref{thm.1.eq.0}), that $t$ belongs to $\tau$. Indeed, since $t$ is a subset of $z$, then $t/x\subset{z/x}$. In addition, as $x$ does not belong to $\varepsilon_x^+(z)$, $\varepsilon_x^+(z)\setminus\{x\}$ is not a subset of $t\setminus\{x\}$. As the homogeneous contraction at vertex $x$ induces a bijection from $\{0,1\}^4\setminus\{x\}$ onto $\{0,1\}^4/x$, this proves that $\varepsilon_x^+(z)/x$ is not a subset of $t/x$ and, following (\ref{thm.1.eq.0}), that $t/x$ belongs to $\tau/x$. Since $\tau/x$ is a subset of $\link_T(\{x\})/x$, one obtains from the definition of $\tau$ that $t\in\tau$.

It has been proven in the last paragraph that $\tau$ contains all the subsets $t\subset{z}$ so that $\varepsilon_x^+(z)\not\subset{t}$. As $\tau\subset{T}$, then the convex hulls of its elements collectively form a polyhedral complex. According to the definition of the triangulations of a circuit, this proves that $\tau$ is the following triangulation of $z$:
$$
\tau=\{t\subset{z}:\varepsilon_x^+(z)\not\subset{t}\}\mbox{.}
$$

In particular, $T$ admits a triangulation of $z$ as a subset and contains $\varepsilon_x^-(z)$. Now consider two maximal faces $s$ and $t$ of $\tau$. It follows that $\varepsilon_x^-(z)$ is a subset of $s$ and $t$. Hence, these two faces contain $x$. In particular, their links in $T$ are subsets of $\link_T(\{x\})$. Since the homogeneous contraction at point $x$ induces an isomorphism from $\link_T(\{x\})$ onto $\link_T(\{x\})/x$, then the links of $s/x$ and $t/x$ in $\link_T(\{x\})/x$ are respectively $\link_T(s)/x$ and $\link_T(t)/x$. In addition, $s/x$ and $t/x$ are maximal faces of $\tau/x$. As $z/x$ is flippable in $\link_T(\{x\})/x$, it follows that $\link_T(s)/x=\link_T(t)/x$. Finally, as the homogeneous contraction at vertex $x$ induces a bijection from $\{0,1\}^4\setminus\{x\}$ onto $\{0,1\}^4/x$, then $\link_T(s)$ and $\link_T(t)$ are necessarily equal, which completes the proof. \qed
\end{proof}

As was the case for lemma \ref{lem.4}, this theorem can be generalized without much more effort to any point configuration of arbitrary dimension. The dimension of $\link_T(\{x\})/x$ and the number of its vertices have to remain small, though, for this result to provide benefit from a computational point of view.
   
Let $x$ be a vertex of $\{0,1\}^4$ and $T$ a triangulation of $\{0,1\}^4$. Call $V_x(T)$ the vertex set of $\link_T(\{x\})/x$ and denote by $D_x(T)$ the domain of $\mathrm{star}_T(\{x\})$: 
$$
D_x(T)=\dom(\mathrm{star}_T(\{x\}))\mbox{.}
$$

Consider a circuit $z\in{Z(x)}$ so that $z$ is flippable in $T$ and $\varepsilon_x^-(z)$ belongs to $T$. Note that these conditions correspond to the conclusions of theorem \ref{thm.1}. Under these conditions, the following theorem states that the domain of $\mathrm{star}_T(\{x\})$ necessarily becomes smaller when circuit $z$ is flipped in $T$ and that, in addition, no vertex is introduced in $\link_T(\{x\})$ by this flip.

\begin{theorem}\label{thm.2}
Let $x$ be a vertex of $\{0,1\}^4$ and $T$ a triangulation of $\{0,1\}^4$. If a circuit $z\in{Z(x)}$ is flippable in $T$ and if $\varepsilon_x^-(z)\in{T}$, then $D_x(\mathfrak{F}(T,z))$ is a proper subset of $D_x(T)$ and $V_x(\mathfrak{F}(T,z))$ is a subset of $V_x(T)$.
\end{theorem}
\begin{proof}
Assume that some circuit $z\in{Z(x)}$ is flippable in $T$ and that $\varepsilon_x^-(z)$ belongs to $T$. Therefore, every face removed from $T$ when flipping $z$ in $T$ admits $\varepsilon_x^-(z)$ as a subset. Hence, these faces necessarily contain $x$ and as a consequence, the union $U$ of their convex hulls is a subset of $D_x(T)$. As the convex hull of any simplex introduced in $T$ by the same flip is a subset of $U$, one therefore obtains that $D_x(\mathfrak{F}(T,z))\subset{D_x(T)}$.

Further denote $t=z\setminus\{x\}$. As $x$ does not belong to $t$, then $\varepsilon_x^-(z)$ is not a subset of $t$. By definition of the triangulations of a circuit, $t$ is a face of $\mathfrak{F}(T,z)$. In addition, if the convex hull of $t$ is a subset of $D_x(\mathfrak{F}(T,z))$, then $t$ necessarily belongs to the link of $\{x\}$ in $\mathfrak{F}(T,z)$. In this case, according to the definition of the link, $t\cup\{x\}$ must be a face of $\mathfrak{F}(T,z)$. As $t\cup\{x\}$ is affinely dependent, this is impossible. One thus obtains an indirect proof that $\conv(t)$ is not a subset of $D_x(\mathfrak{F}(T,z))$. Further observe that $\conv(t)\subset{U}$. As $U$ is, in turn, a subset of $D_x(T)$, this proves that $D_x(\mathfrak{F}(T,z))$ is a proper subset of $D_x(T)$.

Now observe that $V_x(\mathfrak{F}(T,z))$ is always a subset of $V_x(T)$ except if some point is introduced into the link of $x$ when $z$ is flipped in $T$. This occurs only if for some $y\in\{0,1\}^4\setminus\{x\}$, either $\varepsilon_x^+(z)=\{y\}$ or $\varepsilon_x^+(z)=\{x,y\}$. These situations are both impossible, though. Indeed, following its definition, $\varepsilon_x^+(z)$ does not contain $x$. In addition, if $\varepsilon_x^+(z)$ is a singleton then, the unique point it contains must lie in the interior of $[0,1]^4$, which cannot occur because every point in $\varepsilon_x^+(z)$ is a vertex of $[0,1]^4$. As a consequence, $V_x(\mathfrak{F}(T,z))$ is indeed a subset of $V_x(T)$ and the result follows. \qed
\end{proof}

Let $x$ be a point in $\{0,1\}^4$. Further consider a set $S\subset\{0,1\}^4/x$ that admits $\kappa(x)\setminus\{x\}$ as a subset and denote by $L_x(S)$ the set of all the triangulations $T$ of $S$ so that for every circuit $z\in{Z(x)}$, either $z/x$ is not flippable in $T$, or $\varepsilon_x^-(z)/x$ does not belong to $T$. The following result is obtained as a consequence of theorems \ref{thm.1} and \ref{thm.2}:

\begin{corollary}\label{cor.1}
Let $T$ be a triangulation of $\{0,1\}^4$. For any $x\in\{0,1\}^4$, there exists a path in $\gamma(\{0,1\}^4)$ from $T$ to a triangulation $T'$ so that $\link_{T'}(\{x\})/x$ belongs to $L_x(V_x(T))$, and for all $t\in{T}$, if $x\not\in{t}$ then $t\in{T'}$.
\end{corollary}
\begin{proof}
Let $x$ be a vertex of $\{0,1\}^4$. Consider the directed graph $\Gamma$ whose vertices are the triangulations of $\{0,1\}^4$ and whose arcs connect a triangulation $U$ to a triangulation $V$ whenever there exists a circuit $z\in{Z(x)}$ flippable in $U$ so that $\varepsilon_x^-(z)\in{U}$ and $V=\mathfrak{F}(U,z)$. As $\{0,1\}^4$ admits a finite number of triangulations, one can find a triangulation $T'$ of $\{0,1\}^4$ so that there exists a directed path in $\Gamma$ from $T$ to $T'$ and $D_x(T')$ is minimal for the inclusion.

If $\link_{T'}(\{x\})/x$ does not belong to $L_x(V_x(T))$ then, by definition of this set, there exists a circuit $z\in{Z(x)}$ so that $z/x$ is flippable in $\link_{T'}(\{x\})/x$ and $\varepsilon_x^-(z)/x$ is a face of $\link_{T'}(\{x\})/x$. It then follows from theorem \ref{thm.1} that circuit $z$ is flippable in $T'$ and that $\varepsilon_x^-(z)$ belongs to $T'$. Moreover, according to theorem \ref{thm.2}, $D_x(\mathfrak{F}(T',z))$ is a proper subset of $D_x(T')$, which produces a contradiction. Hence, triangulation $\link_{T'}(\{x\})/x$ belongs to $L_x(V_x(T))$.

Now, according to the way $\Gamma$ has been oriented, whenever a flip is performed in a triangulation $U$ on the path that connects $T$ to $T'$, this flip only removes faces of $U$ that contain $x$. As a consequence, all the faces of $T$ that do not contain $x$ are still found in $T'$. \qed
\end{proof}

Let $s$ be an element of $\{E,O\}$, $x$ a vertex of $s$, and $T$ a triangulation of $\{0,1\}^4$. Observe that $\kappa(x)$ is a face of $T$ if and only if $\link_{T}(\{x\})/x$ is the unique triangulation of $\kappa(x)\setminus\{x \}$ (whose unique $3$-dimensional simplex is $\kappa(x)\setminus\{x \}$ itself). It then follows from corollary \ref{cor.1} that if $\kappa(x)\not\in{T}$ and $L_x(V_x(T))$ is a singleton whose element is the unique triangulation of $\kappa(x)\setminus\{x\}$, then one can use a sequence of flips to increase by one the number of corner simplices of $\{0,1\}^4$ with apex in $s$ that are contained in $T$. The success of this method is conditioned by the content of $L_x(V_x(T))$, though. The first part of the next section is devoted to the enumeration of $L_x(S)$ for well chosen subsets $S$ of $\{0,1\}^4/x$. It will then be shown that corollary \ref{cor.1} indeed provides a way to flip any triangulation of $\{0,1\}^4$ to a corner-cut triangulation.  

\section{The flip-graph of $\{0,1\}^4$ is connected}
\label{se.5}

The results of the two previous sections will now be used in order to establish the connectedness of $\gamma(\{0,1\}^4)$. As discussed at the end of last section, corollary \ref{cor.1} gives conditions under which a sequence of flips can be used to increase by one the number of corner simplices found in a triangulation of $\{0,1\}^4$. This section is devoted to showing that these conditions always hold for triangulations of $\{0,1\}^4$ that are not corner-cut.

\begin{algorithm}
\caption{Enumeration of $L_x(S)$.}\label{alg.1}
\begin{algorithmic}
\STATE $L_x(S)\gets\emptyset$
\FOR{every triangulation $T$ of $S$}
\STATE $f_1\gets0$
\FOR{every circuit $z\in{Z(x)/x}$}
\STATE $\lambda\gets\emptyset$
\STATE $f_2\gets1$
\FOR{every point $y$ in $\varepsilon_x^+(z)$}
\IF{$z\setminus\{y\}\in{T}$}
\IF{$\lambda==\emptyset$}
\STATE $\lambda\gets\link_T(z\setminus\{y\})$
\ELSE
\IF{$\lambda\neq\link_T(z\setminus\{y\})$}
\STATE $f_2\gets0$
\ENDIF
\ENDIF
\ELSE
\STATE $f_2\gets0$
\ENDIF
\ENDFOR
\IF{$f_2==1$}
\STATE $f_1\gets1$
\ENDIF
\ENDFOR
\IF{$f_1==0$}
\STATE $L_x(S)\gets{L_x(S)\cup\{T\}}$
\ENDIF
\ENDFOR
\end{algorithmic}
\end{algorithm}
Consider a vertex $x$ of $\{0,1\}^4$ and a point configuration $S\subset\{0,1\}^4/x$ that admits $\kappa(x)\setminus\{x\}$ as a subset. If $S$ is small enough, TOPCOM \cite{Ram02} can be used to obtain a complete list of the triangulations of $S$, as well as the list of the circuits $z\subset{S}$. Once these are known, it is possible to enumerate $L_x(S)$ using algorithm \ref{alg.1}. This algorithm takes advantage of the following property. Consider a circuit $z$ and denote its Radon partition by $\{z^-,z^+\}$. Further denote by $\tau$ the triangulation of $z$ that contains $z^-$. In this case, $\tau$ is a subset of a triangulation $T$ if and only if for all $y\in{z^+}$, $T$ contains $z\setminus\{y\}$. In fact, the sets $z\setminus\{y\}$ are precisely the maximal elements of $\tau$, which further helps to check whether $z$ is flippable in $T$. Observe that algorithm \ref{alg.1} is linear in the number of triangulations of $S$ and that it does not require any arithmetic calculations. This algorithm therefore displaces the difficulty to completely enumerating the triangulations of $S$. If $S=\{0,1\}^4/x$, such an enumeration turns out to be almost as difficult as the enumeration of all the triangulations of $\{0,1\}^4$. Fortunately, the study can be narrowed down to smaller subsets of $\{0,1\}^4/x$. For instance, consider the three following subsets of $\{0,1\}^4/a$:
$$
\left\{
\begin{array}{l}
S_1=[\{0,1\}^4\setminus\{h,n,o,p\}]/a\mbox{,}\\
S_2=[\{0,1\}^4\setminus\{d,j,k,p\}]/a\mbox{,}\\
S_3=[\{0,1\}^4\setminus\{f,g,j,k,p\}]/a\mbox{.}\\
\end{array}
\right.
$$

Observe that $S_1$ is obtained by removing from $\{0,1\}^4/a$ the centroid $p/a$ of tetrahedron $\{b,c,e,i\}/a$ and the centroids of three triangular faces of this tetrahedron. The other two point configurations $S_2$ and $S_3$ are obtained by removing $p/a$ from $\{0,1\}^4/a$ as well as the centroids of respectively three and four edges of $\{b,c,e,i\}/a$. According to the next lemma, if a triangulation $T$ of $\{0,1\}^4$ is not corner-cut, then it is always possible to find a vertex $x$ of $\{0,1\}^4$ and an integer $q\in\{1,2,3\}$ so that $\kappa(x)$ does not belong to $T$ and $V_x(T)$ is isometric to a subset of $S_q$. This result is found as a consequence of lemmas \ref{lem.2} and \ref{lem.3} proven in section \ref{se.3}.

\begin{lemma}\label{lem.5}
Let $T$ be a triangulation of $\{0,1\}^4$ and $s\in\{E,O\}$. If $T$ is not corner-cut and every vertex of $s$ is isolated in $\sigma_4(T)$, then there exist $x\in{s}$ and $q\in\{1,2,3\}$ so that $\kappa(x)\not\in{T}$ and $V_x(T)$ is isometric to a subset of $S_q$.
\end{lemma}
\begin{proof}
Assume that $T$ is not corner-cut and that every vertex of $s$ is isolated in $\sigma_4(T)$. One can check using figure \ref{fig.1} that for any point $x$ in $\{0,1\}^4$, $V_x(T)$ is isometric to a subset of $S_1$ if and only if $x$ is isolated in $\sigma_4(T)$ and has degree at most $1$ in $\sigma_3(T)$. According to lemma \ref{lem.2}, at least four vertices of $s$ have degree at most $1$ in $\sigma_3(T)$. If in addition one of these vertices is not the apex of a corner simplex of $\{0,1\}^4$ found in $T$, then one can find $x\in{s}$ so that $\kappa(x)$ does not belong to $T$ and $V_x(T)$ is isometric to a subset of $S_1$, proving that the desired property holds.

Assume that all the points $x\in{s}$ with degree at most $1$ in $\sigma_3(T)$ are so that $\kappa(x)\in{T}$. According to lemma \ref{lem.2}, then $T$ contains at least four corner simplices of $\{0,1\}^4$ with apex in $s$. In this case, it follows from lemma \ref{lem.3} that all the vertices of $s$ have degree at most $3$ in $\sigma_2(T)$. Now observe that if a point $x\in{s}$ has degree $1$ or $2$ in $\sigma_2(T)$ then $V_x(T)$ is isometric to a subset of $S_2$ or to a subset of $S_3$. Since $T$ is not corner-cut, the wanted result therefore holds when all the vertices of $s$ have degree at most $2$ in $\sigma_2(T)$.

Assume that some vertex of $s$ has degree $3$ in $\sigma_2(T)$. Using the symmetries of $\{0,1\}^4$, one can require that this point be $a$. It can be seen using figure \ref{fig.1} that, if $y$ is one of the three vertices adjacent to $a$ in $\sigma_2(T)$, then $y/a$ is the centroid of an edge of tetrahedron $\kappa(a)\setminus\{a\}$. Denote the set of these three edges by $W$. These edges can be placed in three ways: they can have a common vertex, be the edges of a triangle, or form an acyclic path. These three placements will now be discussed one by one. First. if the three elements of $W$ have a common vertex, one can assume by symmetry that this vertex is $e$. In this case $V_a(T)$ is a subset of $S_2$. Since $a$ is not isolated in $\sigma_2(T)$, then $\kappa(a)$ is not a face of $T$ and the result follows.

The second placement turns out to be impossible. Indeed if the three elements of $W$ are the edges of a triangle, then one can require using the symmetries of $\{0,1\}^4$, that this triangle be $\{b,c,i\}$. In this case, the vertices adjacent to $a$ in $\sigma_2(T)$ are $d$, $j$, and $k$ (see figure \ref{fig.1}). It follows that the points $x\in{s}$ so that $\kappa(x)$ is not a face of $T$ are precisely $a$, $d$, $j$, and $k$. These four points belong to the hyperplane of equation $x\cdot{u_3}=0$. Hence, none of them is a vertex of the $3$-dimensional face of $[0,1]^4$ found in the hyperplane of equation $x\cdot{u_3}=1$. Therefore, the sum of their degrees in $\sigma_3(T)$ is at most $7$ and one of these points has degree at most $1$ in $\sigma_3(T)$. This produces a contradiction with the assumption that every point of $s$ with degree at most $1$ in $\sigma_3(T)$ is the apex of a corner simplex of $\{0,1\}^4$ found in $T$.

Finally, assume that the elements of $W$ are the edges of an acyclic path. By symmetry, one can require that these three edges be $\{b,e\}$, $\{c,i\}$, and $\{e,i\}$. Hence, the vertices adjacent to $a$ in $\sigma_2(T)$ are $f$, $k$, and $m$ (see figure \ref{fig.1}). As a consequence, the points $x\in{s}$ so that $\kappa(x)$ does not belong to $T$ are precisely $a$, $f$, $k$, and $m$. Now observe that $f$ and $k$ have squared distance $4$. As a consequence, $f$ is not adjacent to $k$ in $\sigma_2(T)$, and its degree in $\sigma_2(T)$ is at most $2$. As already discussed above, this implies that $V_f(T)$ is isometric to a subset of $S_2$ or to a subset of $S_3$. As $f$ is not isolated in $\sigma_3(T)$, then $\kappa(f)$ is not a face of $T$, and the desired result follows. \qed
\end{proof}

Point configurations $S_1$, $S_2$, and $S_3$ are small enough to allow for a fast enumeration of their triangulations using TOPCOM: $S_1$ admits $4\,494$ triangulations partitioned into $842$ symmetry classes, $S_2$ admits $3\,214$ triangulations partitioned into $628$ symmetry classes, and $S_3$ admits $596$ triangulations partitioned into $118$ symmetry classes. Up to symmetries, this amounts to a total of $1\,588$ triangulations. Algorithm \ref{alg.1} can then be used to find all the elements of $L_a(S_q)$ for every $q\in\{1,2,3\}$. It follows from the definition of these three sets that the unique triangulation $U_0$ of $\kappa(a)\setminus\{a\}$ belongs to $L_a(S_1)$, to $L_a(S_2)$, and to $L_a(S_3)$. It turns out that $L_a(S_1)$ and $L_a(S_2)$ both contain two more triangulations. Consider the triangulation $U_1^-$ of $\{0,1\}^4/a$ whose maximal faces are the nine following tetrahedra:
$$
\begin{array}{lll}
\{b,c,f,l\}/a, & \{b,f,l,m\}/a, & \{b,i,l,m\}/a,\\
\{c,f,g,l\}/a, & \{c,g,i,l\}/a, & \{e,f,g,l\}/a,\\
\{e,f,l,m\}/a, & \{e,g,l,m\}/a, & \{g,i,l,m\}/a\mbox{.}
\end{array}
$$

This triangulation is depicted in the left of figure \ref{fig.2}. One can see that circuits $\{b,c,f,g\}$, $\{b,f,i,m\}$, and $\{c,g,i,m\}$ are independently flippable in $U_1^-$. Denote by $U_1^+$ the triangulation obtained from $U_1^-$ by flipping these circuits. Triangulations $U_1^-$ and $U_1^+$ are symmetric with respect to the affine hull of $\{b,e,l\}$ and, therefore, they are isometric. The following result can be obtained within a few minutes using TOPCOM and algorithm \ref{alg.1} (see the supplementary material to this article for several detailed implementations):

\begin{proposition}\label{pr.5}
The following statements hold:
\begin{itemize}
\item[i.] $L_a(S_1)$ and $L_a(S_2)$ are equal to $\{U_0,U_1^-,U_1^+\}$,
\item[ii.] $U_0$ is the unique triangulation found in $L_a(S_3)$.
\end{itemize}

\end{proposition}

Let $s$ be an element of $\{E,O\}$. Consider a triangulation $T$ of $\{0,1\}^4$ and a point $x\in{s}$ so that $\kappa(x)$ does not belong to $T$. Further assume that for some integer $q\in\{1,2,3\}$, $V_x(T)$ is isometric to a subset of $S_q$. Following corollary \ref{cor.1} and proposition \ref{pr.5}, $T$ can be flipped to a triangulation $T'$ that contains all the faces of $T$ that do not admit $x$ as a vertex and so that $\link_{T'}(\{x\})/x$ is isometric to either $U_0$ or $U_1^-$. Observe that for all $y\in{s\setminus\{x\}}$, point $x$ is not a vertex of $\kappa(y)$. As a consequence, $T'$ contains all the corner simplices of $\{0,1\}^4$ with apex in $s\setminus\{x\}$ that are already found in $T$. If in addition, $\link_{T'}(\{x\})/x$ is isometric to $U_0$, then $\kappa(x)\in{T'}$ and the sequence of flips that transforms $T$ into $T'$ increases  by one the number of corner simplices of $\{0,1\}^4$ with apex in $s$ found in the triangulation.
\begin{figure}
\begin{centering}
\includegraphics{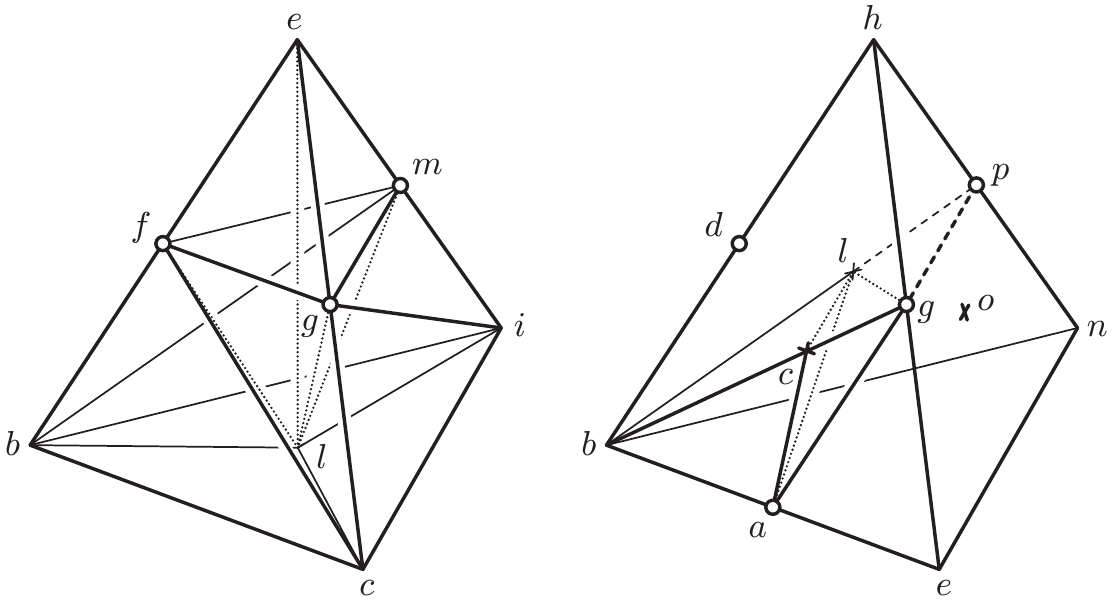}
\caption{Triangulation $U_1^-$ (left), and tetrahedra $\{a,b,c,l\}/f$ and $\{a,b,g,l\}/f$ depicted within point configuration $\{0,1\}^4/f$ (right). Every point has been labeled according to the vertex of $\{0,1\}^4$ it is projected from by the homogeneous contraction. The dashed lines sketch the intersection of $\conv(\{b,e,h,n\}/f)$ with  the affine hull of $\{b,c,g,l\}/f$.}\label{fig.2}
\end{centering}
\end{figure}
This property fails, though, when $\link_{T'}(\{x\})/x$ is not isometric to $U_0$ but to $U_1^-$. A careful study of the triangulations $T$ of $\{0,1\}^4$ so that $\link_T(\{x\})/x$ is isometric to $U_1^-$ is a first step towards solving this problem. Thanks to the symmetries of $[0,1]^4$, it can be assumed without loss of generality that $x$ is equal to $a$ and that $\link_T(\{x\})/x$ is precisely triangulation $U_1^-$. The following lemma, that will be invoked in the proof of theorem \ref{thm.3}, further requires that $T$ contains $\kappa(d)$. This additional condition will also turn out not to be restrictive.

\begin{lemma}\label{lem.6}
Let $T$ be a triangulation of $\{0,1\}^4$. If $\link_T(\{a\})/a$ is equal to $U_1^-$ and $\kappa(d)$ is a face of $T$, then there exist a circuit $z_1\in{Z(a)}$ flippable in $T$ and a circuit $z_2\in{Z(a)}$ flippable in $\mathfrak{F}(T,z_1)$ so that:
\begin{itemize}
\item[i.] $V_a(\mathfrak{F}(\mathfrak{F}(T,z_1),z_2))$ is obtained removing $f$ from $V_a(T)$,
\item[ii.] For all $t\in{T}$, if $\{a,f\}\cap{t}=\emptyset$, then $t\in\mathfrak{F}(\mathfrak{F}(T,z_1),z_2)$.
\end{itemize}
\end{lemma}
\begin{proof}
Assume that $T$ contains $\kappa(d)$ and that $\link_T(\{a\})$ is precisely equal to $U_1^-$. Consider the two following circuits:
$$
z_1=\{b,c,f,g\}\mbox{ and }z_2=\{a,b,e,f\}\mbox{.}
$$

It will be shown that $z_1$ and $z_2$ are flippable in this order in $T$. Observe first that $z_1/f=\{b,c,g\}/f$ and that $z_2/a=\{b,e,f\}/a$. Now consider triangulation $U_1^-$, depicted in the left of figure \ref{fig.2}. One can see that tetrahedra $\{b,c,f,l\}/a$ and $\{c,f,g,l\}/a$ both belong to $\link_T(\{a\})/a$. By definition of the link,
$$
\{a,b,c,f,l\}\in{T}\mbox{ and }\{a,c,f,g,l\}\in{T}\mbox{.}
$$

As a direct consequence, tetrahedra $\{a,b,c,l\}/f$ and $\{a,c,g,l\}/f$ belong to $\link_T(\{f\})/f$. These two tetrahedra are shown in the right of figure \ref{fig.2} within $\{0,1\}^4/f$. One can see in this figure that, according to the definition of flips, if tetrahedra $\{b,c,h,l\}/f$ and $\{c,g,h,l\}/f$ both belong to triangulation $\link_T(\{f\})/f$ then circuit $z_1/f$ is flippable in this triangulation. It turns out that these two tetrahedra actually belong to $\link_T(\{f\})/f$. Indeed, denote by $H$ the affine hull of $\kappa(f)\setminus\{f\}$. Consider the the open half-space of $H$ that contains $h/f$ and that is bounded by the affine hull of $\{b,c,g,l\}/f$ (see the right of figure \ref{fig.2}). Observe that the only element of $\{0,1\}^4/f$ other than $h/f$ found in this open half space is $d/f$. As $\kappa(d)\in{T}$, though, it follows from proposition \ref{pr.2} that edge $\{d,f\}$ does not belong to $T$ and that $d/f$ is not a vertex of $\link_T(\{f\})/f$. 

This proves that the two tetrahedra in $\link_T(\{f\})/f$ that admit $\{b,c,l\}/f$ and $\{c,g,l\}/f$ as respective subsets and that do not admit $a/f$ as their last vertex necessarily both contain point $h/f$. As discussed above, it follows that circuit $z_1/f$ is flippable in $\link_T(\{f\})/f$. Now observe that the Radon partition of circuit $z_1$ is $\{\{b,g\},\{c,f\}\}$. Further note that $\varepsilon_f^-(z_1)=\{c,f\}$. Hence, $\varepsilon_f^-(z_1)/f$ is equal to $\{c/f\}$. As this singleton is contained in $\link_T(\{f\})/f$, it follows from theorem \ref{thm.1} that circuit $z_1$ is flippable in $T$.

Denote $T'=\mathfrak{F}(T,z_1)$. Observe that $\link_{T'}(\{a\})/a$ is obtained by flipping circuit $\{b,c,f,g\}/a$ in $\link_T(\{a\})/a$. One can therefore see in figure \ref{fig.2} that circuit $z_2/a$, that was not flippable in $\link_T(\{a\})/a$ becomes flippable in $\link_{T'}(\{a\})/a$. Moreover the Radon partition of $z_2$ is $\{\{a,f\},\{b,e\}\}$ and $\varepsilon_a^-(z_2)=\{a,f\}$. Hence, $\varepsilon_a^-(z_2)/a$ is equal to $\{f/a\}$. Since $\link_{T'}(\{a\})/a$ contains this singleton, it follows from theorem \ref{thm.1} that circuit $z_2$ is flippable in $T'$. Observe that flipping $z_2$ in $T'$ removes edge $\{a,f\}$. As $f/a$ is the only vertex that has been removed from $V_a(T)$ by the two consecutive flips, the first assertion in the statement of the lemma holds. Further observe that all the faces removed when flipping $z_1$ in $T$ and then $z_2$ in $T'$ contain point $a$ or point $f$. As a consequence, the second assertion also holds. \qed
\end{proof}

Thanks to lemma \ref{lem.6}, it is possible to avoid being stuck by a triangulation isometric to $U_1^-$ on the way to introduce a new corner simplex into a triangulation of the $4$-dimensional cube:

\begin{theorem}\label{thm.3}
Let $T$ be a triangulation of $\{0,1\}^4$, $s$ an element of $\{E,O\}$, and $x$ a vertex of $s$. If $\link_T(\{x\})/x$ is isometric to $U_1^-$ or to $U_1^+$, then there exists a path in $\gamma(\{0,1\}^4)$ that connects $T$ to a triangulation $T'$ so that $\kappa(x)\in{T'}$ and for all $t\in{T}$, if $t\cap{s}=\emptyset$ then $t\in{T'}$.
\end{theorem}
\begin{proof}
Assume that $\link_T(\{x\})/x$ is isometric to $U_1^-$ or to $U_1^+$. By symmetry, one can further assume that $x=a$ and that $\link_T(\{x\})/x=U_1^-$. In this case, $s$ is equal to $E$. Looking at the sketch of triangulation $U_1^-$, in the left of figure \ref{fig.2}, one obtains from the definition of $\link_T(\{a\})/a$ that edges $\{a,l\}$, $\{c,f\}$, $\{f,l\}$, and $\{g,l\}$ belong to $T$. Figure \ref{fig.3} shows these four edges placed into the $3$-dimensional cubical faces of $[0,1]^4$ they are diagonals of. These cubes are precisely the $3$-dimensional faces of $[0,1]^4$ that admit $d$ as a vertex. Since $d$ is not a vertex of any of the diagonals shown in this figure, and since $T$ cannot contain two diagonals of a same cube then $d$ is isolated in $\sigma_3(T)$.
 
 \begin{figure}
\begin{centering}
\includegraphics{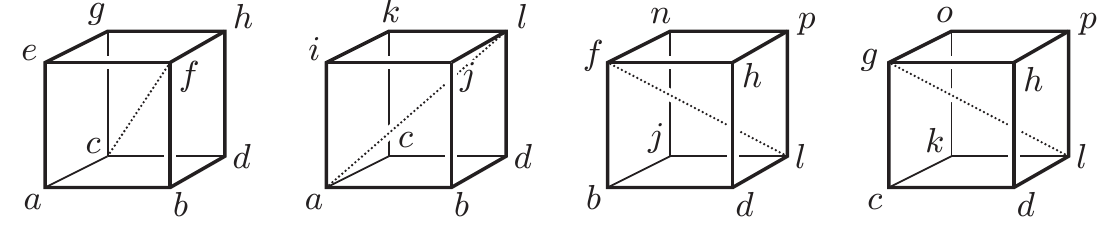}
\caption{The four $3$-dimensional cubical faces of $[0,1]^4$ that admit $d$ as a vertex. Each of these cubes has one of its diagonals in $U_1^-$, depicted here as a dotted line.}\label{fig.3}
\end{centering}
\end{figure}

Further observe that $\{e,l\}$ is an edge of graph $\sigma_4(T)$ (indeed, as can be seen in figure \ref{fig.2}, $\link_T(\{a\})/a$ contains $\{e,l\}/a$). It then follows from proposition \ref{pr.2} that $e$ and $l$ are the only two non-isolated vertices of $\sigma_4(T)$. In particular, $d$ is not only isolated in $\sigma_3(T)$, but also in $\sigma_4(T)$. This proves that $V_d(T)$ is isometric to a subset of $S_1\setminus\{l/a\}$. Now observe that $L_a(S_1\setminus\{l/a\})$ is a subset of $L_a(S_1)$. Following proposition \ref{pr.5}, the only triangulation of $S_1\setminus\{l/a\}$ contained in $L_a(S_1)$ is $U_0$. As a consequence, this triangulation is the unique element of $L_a(S_1\setminus\{l/a\})$. Since $V_d(T)$ is isometric to a subset of $S_1\setminus\{l/a\}$, one obtains that $L_d(V_d(T))$ is a singleton whose element is the unique triangulation of $\kappa(d)\setminus\{d\}$. Hence, according to corollary \ref{cor.1}, $T$ can be flipped to a triangulation $T''$ so that $\kappa(d)\in{T''}$ and for all $t\in{T}$, if $d\not\in{t}$ then $t\in{T''}$.

As $\{a,d\}$ is not an edge of $T$ (see the left of figure \ref{fig.2}), the star of $\{a\}$ in $T$ is not affected by the flips that transform $T$ into $T''$. In particular, $\link_{T''}(\{a\})/a$ is equal to $U_1^-$. Hence $T''$ satisfies the conditions of lemma \ref{lem.6}, and there exists a path in $\gamma(\{0,1\}^4)$ from triangulation $T''$ to a triangulation $T'$ so that $\kappa(a)\in{T'}$ and for all $t\in{T''}$, if $\{a,f\}\cap{t}$ is empty then $t\in{T'}$. It follows that all the elements of $T$ that do not contain $a$, $d$, or $f$ are still found in $T''$. As these three points belong to $s$, $T'$ has the desired properties. \qed
\end{proof}

According to proposition \ref{pr.5}, the next result is an immediate consequence of corollary \ref{cor.1} and of theorem \ref{thm.3}:

\begin{corollary}\label{cor.2}
Let $T$ be a triangulation of $\{0,1\}^4$, $s$ an element of $\{E,O\}$, and $x\in{s}$ a point. If for some $q\in\{1,2,3\}$, $V_x(T)$ is isometric to a subset of $S_q$, then there exists a path in $\gamma(\{0,1\}^4)$ that connects $T$ to a triangulation $T'$ so that $\kappa(x)$ belongs to $T'$ and for all $t\in{T}$, if $t\cap{s}=\emptyset$ then $t\in{T'}$.
\end{corollary}
\begin{proof}
Assume that there exists $q\in\{1,2,3\}$ so that $V_x(T)$ is isometric to a subset of $S_q$. According to corollary \ref{cor.1}, there exists a path in $\gamma(\{0,1\}^4)$ from $T$ to a triangulation $T''$ so that $\link_{T''}(\{x\})/x\in{L_x(V_x(T))}$ and for all $t\in{T}$, if $x\not\in{t}$ then $t\in{T''}$. If then follows from proposition \ref{pr.5} that $\link_{T''}(\{x\})/x$ is isometric to either $U_1^-$ or to $U_0$. In the latter case, $x$ is necessarily isolated in $\sigma_2(T'')$, in $\sigma_3(T'')$, and in $\sigma_4(T'')$ and according to proposition \ref{pr.3}, $\kappa(x)\in{T''}$. As in addition all the faces of $T$ that do not contain $x$ are necessarily found in $T''$, then the wanted properties hold by taking $T'=T''$.

Now assume that $\link_{T''}(\{x\})/x$ is isometric to $U_1^-$. In this case, theorem \ref{thm.3} provides a path in $\gamma(\{0,1\}^4)$ connects $T''$ to a triangulation $T'$ so that $\kappa(x)$ is a face of $T'$ and for all $t\in{T''}$, if $t\cap{s}=\emptyset$, then $t\in{T'}$. Since $x\in{s}$, then every face of $T$ disjoint from $s$ still belong to $T'$, which completes the proof. \qed
\end{proof}

Combining lemma \ref{lem.5}, and corollary \ref{cor.2}, one finds that every triangulation of $\{0,1\}^4$ can be transformed into a corner-cut triangulation by a sequence of flips. This leads to the main result of this paper:

\begin{theorem}\label{thm.4}
The flip-graph of $\{0,1\}^4$ is connected.
\end{theorem}
\begin{proof}
Let $T$ be a triangulation of $\{0,1\}^4$ and $s$ an element of $\{E,O\}$. Consider the triangulations that are connected to $T$ by a path in $\gamma(\{0,1\}^4)$ and that contain all the faces of $T$ disjoint from $s$. Among these triangulations, let $T''$ be one that contains the largest possible number of corner simplices of $\{0,1\}^4$ with apex in $s$.

Since $\sigma_4(T'')$ has at most one edge, and since the squared length of such an edge is even, then either all the points of $s$ are isolated in $\sigma_4(T'')$ or exactly two of them are adjacent. If all the points of $s$ are isolated in $\sigma_4(T'')$ then denote $T'=T''$ and $s'=s$. Otherwise, let $s'$ be the element of $\{E,O\}$ distinct from $s$, and consider the triangulations that are connected to $T''$ by a path in $\gamma(\{0,1\}^4)$ and that contain all the faces of $T''$ disjoint from $s'$. Among these triangulations, let $T'$ be one that contains the largest possible number of corner simplices of $\{0,1\}^4$ with apex in $s'$. In this case, the vertices of $s'$ are necessarily isolated in $\sigma_4(T')$. Indeed, $\sigma_4(T'')$ contains an edge whose vertices are not in $s'$. Therefore this edge still belongs to $\sigma_4(T')$ and according to proposition \ref{pr.3}, no other edge is found in $\sigma_4(T')$.

According to this construction, there exists a path in $\gamma(\{0,1\}^4)$ from $T$ to $T'$, so that all the elements of $s'$ are isolated in $\sigma_4(T')$. Moreover, the number of corner simplices of $\{0,1\}^4$ with apex in $s'$ found in $T'$ cannot be increased by a sequence of flips without removing from $T'$ a face disjoint from $s'$. Now, assume that $T'$ is not corner-cut. Since all the vertices of $s'$ are isolated in $\sigma_4(T')$, it follows from lemma \ref{lem.5} that there exists a point $x\in{s'}$ and an integer $q\in\{1,2,3\}$ so that $V_x(T')$ is isometric to a subset of $S_q$ and $\kappa(x)\not\in{T'}$. In particular, $T'$ satisfies the conditions of corollary \ref{cor.2} and one can find a triangulation $T^{(3)}$ connected to $T'$ by a path in $\gamma(\{0,1\}^4)$ so that $T^{(3)}$ contains $\kappa(x)$ and every face of $T'$ disjoint from $s'$. In particular, for any $y\in{s'}$ so that $\kappa(y)\in{T'}$, $T^{(3)}$ contains $\kappa(y)\setminus\{y\}$. As a consequence, $y$ is isolated in graphs $\sigma_2(T^{(3)})$, $\sigma_3(T^{(3)})$, and $\sigma_4(T^{(3)})$, and it follows from proposition \ref{pr.3} that $\kappa(y)\in{T^{(3)}}$. In other words, a sequence of flips was found that increases the number of corner simplices of $\{0,1\}^4$ with apex in $s'$ found in $T'$. Moreover, this sequence of flips does not remove from $T'$ any face disjoint from $s'$.

This produces a contradiction, proving that $T'$ is corner-cut. According to lemma \ref{lem.1}, a path in $\gamma(\{0,1\}^4)$ therefore connects triangulation $T$ to a regular triangulation. The result then follows from the connectedness of the subgraph $\rho(\{0,1\}^4)$ induced by regular triangulations in $\gamma(\{0,1\}^4)$. \qed
\end{proof}

As mentioned in the introduction, the only algorithm efficient enough to completely enumerate the triangulations of $\{0,1\}^4$ consists in exploring the flip-graph of $\{0,1\}^4$, and its validity is based on the connectedness of this graph. Theorem \ref{thm.4} solves this issue. Using TOPCOM, one can therefore perform this enumeration in (much) less than an hour, with the following result:

\begin{corollary}
The vertex set of the $4$-dimensional cube admits $92\,487\,256$ triangulations, partitioned into $247\,451$ symmetry classes.
\end{corollary}

\section{Discussion}

The main result of this paper is the connectedness of $\gamma(\{0,1\}^4)$. Computer assistance was used on the way, in order to obtain proposition \ref{pr.5}. This computer-assisted part consists in checking a property regarding flips over $1\,588$ triangulations of point configuration $\{0,1\}^4/a$. This number is small enough to allow for a fastidious, but possible, verification by hand. However, the computer is certainly more reliable in such a task. In addition, the computer-assisted verification can be carried out within a few minutes. A second important result obtained here is the complete enumeration of the triangulations of $\{0,1\}^4$: $92\,487\,256$ triangulations were found. It was established a few years ago that the number of regular triangulations of $\{0,1\}^4$ is $87\,959\,448$ \cite{Hug08}. Hence, more than $95\%$ of the triangulations of $\{0,1\}^4$ are regular. This does not come as a surprise, since non-regular triangulations of $\{0,1\}^4$ are difficult to find \cite{deL96}.

These observations lead to a natural question. Call a triangulation of a $d$-dimensional point configuration {\it $2$-regular} when its faces are projected from the boundary complex of a $(d+2)$-dimensional polytope. It was found recently that the subgraph induced by $2$-regular triangulations in the flip-graph of a point configuration is connected \cite{Pou12a}, providing new ways to investigate flip-graph connectivity \cite{Pou12b,Pou12c}. Since so few triangulations of $\{0,1\}^4$ are non-regular, it seems natural to conjecture that {\it all the triangulations of $\{0,1\}^4$ are $2$-regular}. Settling this conjecture positively would in addition provide another proof for the connectedness of $\gamma(\{0,1\}^4)$.

\end{document}